\documentclass{amsart}
\usepackage{setspace, amsmath, amsthm, amssymb, amsfonts, amscd, epic, graphicx, ulem, dsfont}
\usepackage[T1]{fontenc}
\usepackage{multirow}
\usepackage{bbm}
\usepackage{enumerate}

\makeatletter \@namedef{subjclassname@2010}{
  \textup{2010} Mathematics Subject Classification}
\makeatother

\newtheorem{thm}{Theorem}[section]

\newtheorem{lem}[thm]{Lemma}
\newtheorem{pro}[thm]{Proposition}
\newtheorem{conj}[thm]{Conjecture}

\theoremstyle{remark}
\newtheorem*{rema}{Remark}

\theoremstyle{definition}

\newcommand{\ran}{\text{\rm{ran}}}

\newcommand{\R}{\mathbb{R}}

\newcommand{\N}{\mathbb{N}}

\begin{document}

\title[Triviality of domains of powers of operators]{On the triviality of domains of powers and adjoints of closed operators}
\author[M. H. MORTAD]{Mohammed Hichem Mortad}

\dedicatory{}
\thanks{}
\date{}


 \address{Department of
Mathematics, University of Oran 1, Ahmed Ben Bella, B.P. 1524, El
Menouar, Oran 31000, Algeria.\newline {\bf Mailing address}:
\newline Pr Mohammed Hichem Mortad \newline BP 7085 Seddikia Oran
\newline 31013 \newline Algeria}

\email{mhmortad@gmail.com, mortad.hichem@univ-oran1.dz.}

\begin{abstract}
The paper is devoted to counterexamples involving the triviality of
domains of products and/or adjoints of densely defined operators.
\end{abstract}

\maketitle

\section{Introduction}

Counterexamples about non necessarily bounded operators have not
stopped to impress us. The striking example due to Chernoff is well
known to specialists. It states that there is a closed, unbounded,
densely defined, symmetric and semi-bounded operator $A$ such that
$D(A^2)=\{0\}$ (see \cite{CH}). This counterexample came in to
simplify a rather complicated construction already obtained by
Naimark in \cite{NAI}. It is worth noticing that Schm\"{u}dgen
\cite{SCHMUDG-1983-An-trivial-domain} obtained almost simultaneously
(as Chernoff) that every unbounded self-adjoint $T$ has two closed
and symmetric restrictions $A$ and $B$ such that
\[D(A)\cap D(B)=\{0\}\text{ and } D(A^2)=D(B^2)=\{0\}.\]

This fascinating result by Schm\"{u}dgen (which was later
generalized by Brasche-Neidhardt in \cite{Brasche-Neidhardt}. See
also \cite{Arlinski-Zagrebnov}) also dealt with higher powers.

Recently, the author obtained (jointly with S. Dehimi) in
\cite{Dehimi-Mortad-Chernoff} a fairly simple example based upon
matrices of unbounded operators. Based on results from
\cite{Dehimi-Mortad-Chernoff}, we propose the following conjecture:

\begin{conj}\label{Conjecture main}
For each $n\in\N$, there is a closed and densely defined $T$ such
that $D(T^{n-1})\neq\{0\}$ and $D[T^{*(n-1)}]\neq\{0\}$ yet
\[D(T^n)=D(T^{*n})=\{0\}.\]
\end{conj}

It is worth emphasizing that even though K. Schm\"{u}dgen obtained
the general case of $n$ powers, here we give more explicit
counterexamples and the novelty is that the counterexamples in our
case concern both the powers of an operator as well as the powers of
their adjoints.

The main aim of this paper is to try to give answers to the previous
conjecture. It is just amazing how matrices of unbounded operators,
despite their unexpected behavior in some cases, can make things
fairly easy to deal with.  The same approach has equally allowed us
to find more interesting counterexamples on a different topic. See
\cite{Mortad-commutators-unboudned-CEXAMP}.

In the end, we assume readers are familiar with notions and results
on unbounded operators and, in particular, matrices of unbounded
operators. We refer readers to \cite{tretetr-book-BLOCK} for
properties of block operator matrices. From some recent papers on
matrices of unbounded operators, we cite
\cite{JIN-CHEN-MATrix-UNBOUNDED-ARXIV},
\cite{Moller-Szafanriac-matri-unbounded} and
\cite{Ota-Schmudgen-2003-Matrices-Operators}. For the general theory
of unbounded operators, readers may wish to consult
\cite{SCHMUDG-book-2012} or \cite{Weidmann}.

\section{Main Counterexamples}

We start with an auxiliary example which is also interesting in its
own.

\begin{pro}\label{AB trivial domain inspired KOSAKI QUES}
There is an unbounded self-adjoint and positive operator $A$ and an
everywhere defined bounded and self-adjoint $B$ such that
$D(AB)=\{0\}$ and $D(BA)\neq\{0\}$ (in fact $D(BA)=D(A)$ is dense).
\end{pro}

\begin{proof}In fact, we have a slightly better counterexample than what
is suggested. Consider the operators $A$ and $B$ defined by
\[Af(x)=e^{\frac{x^2}{2}}f(x)\]
on $D(A)=\{f\in L^2(\R):~e^{\frac{x^2}{2}}f\in L^2(\R)\}$ and
$B:=\mathcal{F}^*A\mathcal{F}$ where $\mathcal{F}$ designates the
usual $L^2(\R)$-Fourier Transform. Then $D(A)\cap D(B)=\{0\}$ (see
e.g. \cite{KOS}). Clearly $A$ is boundedly invertible and
\[A^{-1}f(x)=e^{\frac{-x^2}{2}}f(x)\]
is defined from $L^2(\R)$ onto $D(A)$.

Recall that $D(BA^{-1})$ is trivial if $D(B)\cap \ran(A^{-1})$ is so
and if $A^{-1}$ is one-to-one. That $A^{-1}$ is injective is plain.
Now,
\[D(B)\cap \ran(A^{-1})=D(B)\cap D(A)=\{0\}\] and this
is already available to us.  In the end,
\[D(A^{-1}B)=D(B)\]
which is evidently dense in $L^2(\R)$.
\end{proof}

\begin{pro}
There exists a densely defined operator $T$ such that
\[D(T^*)=D(T^2)=D(TT^*)=D(T^*T)=\{0\}.\]
\end{pro}

\begin{proof}There are known examples in the literature about the case $D(T^*)=\{0\}$. For
instance, Example 3.4 on Page 105 in
\cite{Jorgensen-Tian-Book-non-comm-analysis} or Example 3 on Page 69
in \cite{Weidmann}. These examples are not that straightforward. The
example we are about to give is truly simple.

Consider the operators $A$ and $B$ introduced in Proposition \ref{AB
trivial domain inspired KOSAKI QUES}, that is,
\[Af(x)=e^{\frac{x^2}{2}}f(x)\]
on $D(A)=\{f\in L^2(\R):~e^{\frac{x^2}{2}}f\in L^2(\R)\}$ and
$B:=\mathcal{F}^*A\mathcal{F}$. We then found that
$D(BA^{-1})=\{0\}$.

Now, set $T:=A^{-1}B$. Then $T$ is densely defined because
$D(T)=D(B)$ as also $A^{-1}\in B(L^2(\R))$. Thus,
\[D(T^*)=D[(A^{-1}B)^*]=D(BA^{-1})=\{0\},\]
as needed. Hence plainly
\[D(TT^*)=\{0\}.\]
Now,
\[D(T^*T)=\{f\in D(T):Tf\in D(T^*)\}=\{f\in D(A^{-1}B):A^{-1}Bf=0\}=\{0\}\]
for $A^{-1}B$ is one-to-one. Finally,
\[D(T^2)=D(A^{-1}BA^{-1}B)=D[(BA^{-1})B]=\{f\in D(B):Bf=0\}\]
and so $D(T^2)=\{0\}$ by the injectivity of $B$.
\end{proof}

\begin{pro}
There is a densely defined and closed operator $T$ such that
$D(T^2)\neq\{0\}$ and $D(T^{*2})\neq\{0\}$ but
\[D(T^3)=D(T^{*3})=\{0\}.\]
\end{pro}

\begin{proof}
Let $A$ and $B$ be self-adjoint operators such that
$D(AB)=\{0_{L^2(\R)}\}$ but $D(BA)\neq\{0_{L^2(\R)}\}$ as in
Proposition \ref{AB trivial domain inspired KOSAKI QUES}. Remember
that there $A$ is one-to-one. Now, on $L^2(\R)\oplus L^2(\R)$, set
\[T=\left(
      \begin{array}{cc}
        0 & A \\
        B & 0 \\
      \end{array}
    \right).
\]
Hence
\[T^2=\left(
      \begin{array}{cc}
        0 & A \\
        B & 0 \\
      \end{array}
    \right)\left(
      \begin{array}{cc}
        0 & A \\
        B & 0 \\
      \end{array}
    \right)=\left(
      \begin{array}{cc}
        AB & 0 \\
        0 & BA \\
      \end{array}
    \right)\]
    and so $D(T^2)=\{0_{L^2(\R)}\}\oplus D(BA)\neq
    \{(0_{L^2(\R)},0_{L^2(\R)})\}$. Finally,
    \[T^3=\left(
      \begin{array}{cc}
        AB & 0 \\
        0 & BA \\
      \end{array}
    \right)\left(
      \begin{array}{cc}
        0 & A \\
        B & 0 \\
      \end{array}
    \right)=\left(
      \begin{array}{cc}
        0 & ABA \\
        BAB & 0 \\
      \end{array}
    \right).\]
    Obviously, $D(BAB)=\{0_{L^2(\R)}\}$. Since
    \[D(ABA)=\{x\in D(A):Ax\in D(AB)=\{0_{L^2(\R)}\}\}=\ker A\]
    and $A$ is injective, it follows that we equally have
    $D(ABA)=\{0_{L^2(\R)}\}$. Accordingly,
    $D(T^3)=\{(0_{L^2(\R)},0_{L^2(\R)})\}$. Finally, as
    \[T^*=\left(
      \begin{array}{cc}
        0 & B \\
        A & 0 \\
      \end{array}
    \right),\]
    then we may similarly show that $D(T^{*2})\neq\{0\}$
    and  $D(T^{*3})=\{0\}$, marking the end of the proof.
\end{proof}

\begin{pro}\label{nnnnnnnn}
There exists a densely defined and closed operator $T$ such that
$D(T^3)\neq\{0\}$ and $D(T^{*3})\neq\{0\}$ yet
\[D(T^4)=D(T^{*4})=\{0\}.\]
\end{pro}

\begin{rema}
Obviously, $D(T^3)\neq\{0\}$ will insure that $D(T^2)\neq\{0\}$. The
same remark applies to $D(T^{*3})\neq\{0\}$ and
$D(T^{*2})\neq\{0\}$.
\end{rema}

The counterexample is based on the following recently obtained
result:

\begin{lem}\label{mmm}(\cite{Dehimi-Mortad-Chernoff}) There are unbounded self-adjoint operators $A$ and $B$
such that
\[D(A^{-1}B)=D(BA^{-1})=\{0\}\]
(where $A^{-1}$ and $B^{-1}$ are not bounded).
\end{lem}

Now, we give the proof of Proposition \ref{nnnnnnnn}:

\begin{proof}Let $A$ and $B$ be two unbounded self-adjoint operators
such that
\[D(A^{-1}B)=D(BA^{-1})=\{0\}\]
where $A^{-1}$ and $B^{-1}$ are not bounded. Now, define
\[S=\left(
      \begin{array}{cc}
        0 & A^{-1} \\
        B & 0 \\
      \end{array}
    \right)
\]
on $D(S):=D(B)\oplus D(A^{-1})\subset L^2(\R)\oplus L^2(\R)$. Then
$S$ is densely defined and closed. In addition, we already know from
Lemma \ref{mmm} that $D(A^{-1}B)=D(BA^{-1})=\{0\}$ and so
$D(S^2)=D(S^{*2})=\{0\}$ (as in \cite{Dehimi-Mortad-Chernoff} say).
Notice now that we may write
\[S=\underbrace{\left(
      \begin{array}{cc}
         A^{-1}&0 \\
        0 & B \\
      \end{array}
    \right)}_{C}\underbrace{\left(
      \begin{array}{cc}
        0 & I \\
        I & 0 \\
      \end{array}
    \right)}_{D}\text{ and }S^*=DC\]
because $C$ and $D$ are self-adjoint and $D^{-1}\in B(H)$ ($D$ is
even a fundamental symmetry). Now, define $T$ on $L^2(\R)\oplus
L^2(\R)\oplus L^2(\R)\oplus L^2(\R)$ by
\[T=\left(
      \begin{array}{cc}
        \mathbf{0} & C \\
        D & \mathbf{0} \\
      \end{array}
    \right)
\]
where $\mathbf{0}$ is the zero matrix of operators on $L^2(\R)\oplus
L^2(\R)$. Then,
\[T^2=\left(
      \begin{array}{cc}
        CD & \mathbf{0} \\
        \mathbf{0} & DC \\
      \end{array}
    \right),~T^3=\left(
      \begin{array}{cc}
        \mathbf{0} & CDC \\
        DCD & \mathbf{0} \\
      \end{array}
    \right)\]\text{ and }
    \[T^4=\left(
      \begin{array}{cc}
        CDCD & \mathbf{0}\\
        \mathbf{0}& DCDC \\
      \end{array}
    \right)=\left(
      \begin{array}{cc}
        S^2 & \mathbf{0}\\
        \mathbf{0}& S^{*2} \\
      \end{array}
    \right).\]
Also, since $T^*=\left(
      \begin{array}{cc}
        \mathbf{0} & D \\
        C & \mathbf{0} \\
      \end{array}
    \right)$, we equally have
\[T^{*2}=\left(
      \begin{array}{cc}
        DC & \mathbf{0} \\
        \mathbf{0} & CD \\
      \end{array}
    \right),~T^{*3}=\left(
      \begin{array}{cc}
        \mathbf{0} & DCD \\
        CDC & \mathbf{0} \\
      \end{array}
    \right)\]\text{ and }
    \[T^{*4}=\left(
      \begin{array}{cc}
        DCDC & \mathbf{0}\\
        \mathbf{0}& CDCD \\
      \end{array}
    \right)=\left(
      \begin{array}{cc}
        S^{*2} & \mathbf{0}\\
        \mathbf{0}& S^{2} \\
      \end{array}
    \right).\]

Finally, observe that
\[D(T^2)=D(S)\oplus D(S^*)=D(B)\oplus D(A^{-1})\oplus D(A^{-1})\oplus D(B)\neq \{0_{[L^2(\R)]^4}\},\]
that
\[D(T^3)=D(B)\oplus D(A^{-1})\oplus \{0\}\oplus\{0\}\neq \{0_{[L^2(\R)]^4}\}\]
but
\[D(T^4)=D(S^2)\oplus D(S^{*2})=\{0_{[L^2(\R)]^4}\}.\]
The corresponding relations about the domains of adjoints may be
checked similarly. The proof is therefore complete.
\end{proof}

The same idea of proof may be carried over to higher powers,
however, we have not been able yet to establish a general
counterexample. We give further counterexamples which may inspire
readers to find the coveted general counterexample. For example, we
know how to deal with the case $n=6$.

\begin{pro}There exists a densely defined and closed $T$ such that
$D(T^5)\neq\{0\}$ and $D(T^{*5})\neq\{0\}$ whilst
\[D(T^6)=D(T^{*6})=\{0\}.\]

\end{pro}

\begin{proof}We shall avoid unnecessary details as similar cases have
already been treated. First, choose $A$ and $B$ as in Proposition
              \ref{AB trivial domain inspired KOSAKI QUES}, that is, $D(AB)=\{0\}$ and $D(BA)=D(A)$ where $A$ and $B$ are self-adjoint operators and $B$ is bounded and everywhere defined.

Now, let
\[C=\left(
      \begin{array}{cc}
        B & 0 \\
        0 & A \\
      \end{array}
    \right)\text{ and }D=\left(
                           \begin{array}{cc}
                             0 & I \\
                             I & 0 \\
                           \end{array}
                         \right).
\]
Then $C$ and $D$ are self-adjoint (remember that $D$ is even
everywhere defined and bounded). Set $S=CD=\left(
                \begin{array}{cc}
                  0 & B \\
                  A & 0 \\
                \end{array}
              \right)$. Hence $S^*=\left(
                \begin{array}{cc}
                  0 & A \\
                  B & 0 \\
                \end{array}
              \right)$. Finally,
              define $T$ on $[L^2(\R)]^4$ by
              \[T=\left(
                    \begin{array}{cc}
                      \mathbf{0} & C \\
                      D & \mathbf{0} \\
                    \end{array}
                  \right) \text{ and so } T^*=\left(
                    \begin{array}{cc}
                      \mathbf{0} & D \\
                      C & \mathbf{0} \\
                    \end{array}
                  \right)
              \]
with $\mathbf{0}$ being the zero matrix of operators on
$L^2(\R)\oplus L^2(\R)$. Hence
\[T^6=\left(
        \begin{array}{cc}
          S^3 & \mathbf{0} \\
          \mathbf{0} & S^{*3} \\
        \end{array}
      \right) \text{ and }T^{*6}=\left(
        \begin{array}{cc}
          S^{*3} & \mathbf{0} \\
          \mathbf{0} & S^{3} \\
        \end{array}
      \right).
\]

Accordingly, $D(T^6)=D(S^3)\oplus D(S^{*3})=\{(0,0)\}=D(T^{*6})$
thanks to the assumptions on $A$ and $B$. However, \[T^5=\left(
                    \begin{array}{cc}
                      \mathbf{0} & CDCDC \\
                      DCDCD & \mathbf{0} \\
                    \end{array}
                  \right),~T^{*5}=\left(
                    \begin{array}{cc}
                      \mathbf{0} & DCDCD \\
                      CDCDC & \mathbf{0} \\
                    \end{array}
                  \right)\] and, as can simply be checked, \[D(CDCDC)=\{0\}
                  \text{ but }D(DCDCD)\neq\{0\}.\] Consequently,
                  $D(T^5)\neq\{0\}$. A similar reasoning yields
                  $D(T^{*5})\neq\{0\}$, marking the end of the
                  proof.
\end{proof}

The next counterexamples settles the case of powers of the type
$2^n$ via what we may call "nested matrices".

\begin{pro}
For each $n\in\N$, there is a densely defined and closed operator
$T$ (which is an off-diagonal matrix of operators) such that
$D(T^{2^n-1})\neq\{0\}$ and $D(T^{*{2^n-1}})\neq\{0\}$ whereas
\[D(T^{2^n})=D(T^{*^{2^n}})=\{0\}.\]
\end{pro}

\begin{proof}We use a proof by induction. The statement is true for
$n=2$ as seen before. Assume now that there is a closed $T=\left(
                                                             \begin{array}{cc}
                                                               0 & A \\
                                                               B & 0 \\
                                                             \end{array}
                                                           \right)$
                                                           such that
                                                           $D(T^{2^n-1})\neq\{0\}$
and $D(T^{*{2^n-1}})\neq\{0\}$ with
$D(T^{2^n})=D(T^{*{2^n}})=\{0\}$. Now, write
\[T=\left(
                                                             \begin{array}{cc}
                                                               0 & A \\
                                                               B & 0 \\
                                                             \end{array}
                                                           \right)=\underbrace{\left(
                                                             \begin{array}{cc}
                                                               A & 0 \\
                                                               0 & B \\
                                                             \end{array}
                                                           \right)}_{C}\underbrace{\left(
                                                             \begin{array}{cc}
                                                               0 & I \\
                                                               I & 0 \\
                                                             \end{array}
                                                           \right)}_{\tilde{B}}\]
Next, set $S=\left(
               \begin{array}{cc}
                 0 & C \\
                 \tilde{B} & 0 \\
               \end{array}
             \right)
$ and so
\[S^{2^{n+1}}=\left(
                \begin{array}{cc}
                  (C\tilde{B})^{2^n} & 0 \\
                  0 & (\tilde{B}C)^{2^n} \\
                \end{array}
              \right)=\left(
                        \begin{array}{cc}
                          T^{2^n} & 0 \\
                          0 & T^{*^{2^n}} \\
                        \end{array}
                      \right)
\]
and so $D(S^{2^{n+1}})=\{0\}$. On the other hand,
\[S^{2^{n+1}-1}=\left(
                  \begin{array}{cc}
                    0 & (C\tilde{B})^{2^n-1}C \\
                    \tilde{B}(C\tilde{B})^{2^n-1} & 0 \\
                  \end{array}
                \right).
\]
Since $\tilde{B}$ is everywhere bounded, it follows that
\[D(\tilde{B}(C\tilde{B})^{2^n-1})=D((C\tilde{B})^{2^n-1})=D(T^{2^n-1})\neq\{0\}\]
which leads to $D(S^{2^{n+1}-1})\neq\{0\}$, as wished.

The case of adjoints may be treated similarly and hence we omit it.
\end{proof}

\section{Conclusion}

Even though we have managed to obtain a fair amount of
counterexamples, we have not been able to solve the whole
conjecture. If the conjecture is false for some $n$, then proving
this is not going to be an easy task. So, we invite interested
readers to try to contribute towards a complete answer to the main
problem.

I also take this opportunity to thank Dr S. Dehimi with whom a
fruitful discussion has led to the proof of Proposition \ref{AB
trivial domain inspired KOSAKI QUES}.

\end{document}